\documentclass[12pt]{amsart}

\usepackage{amsmath,amsthm,amssymb}
\font\teneufm=eufm10
\font\seveneufm=eufm7
\font\fiveeufm=eufm5
\newfam\frakturfam

\textfont\frakturfam=\teneufm
\scriptfont\frakturfam=\seveneufm
\scriptscriptfont\frakturfam=\fiveeufm

\theoremstyle{definition}
\newtheorem{definition}{Definition}[section]

\theoremstyle{plain}
\newtheorem{lemma}[definition]{Lemma}
\newtheorem{prp}[definition]{Proposition}

\newtheorem{theorem}[definition]{Theorem}

\begin{document}

\title{On a variety of right-symmetric algebras}

\author{Nurlan Ismailov}

\address{Astana IT University,
Mangilik El avenue, 55/11, Business center EXPO, block C1,
Astana, 010000, Kazakhstan}

\email{nurlan.ismail@gmail.com}

\author{Ualbai Umirbaev}

\address{Department of Mathematics,
 Wayne State University, Detroit, MI 48202, USA; 
and Institute of Mathematics and Mathematical Modeling, Almaty, 050010, Kazakhstan}

\email{umirbaev@wayne.edu}

%\thanks{}

\maketitle

\begin{abstract}
We construct a finite-dimensional metabelian right-symmetric algebra over an arbitrary field that does not have a 
finite basis of identities.

\end{abstract}


\noindent


\section{\label{1}\ Inroduction.}

We say that a variety of algebras has {\em the Specht property} or is {\em Spechtian} if any of its subvarieties has a finite basis of identities. In other words, a variety of algebras is Spechtian if the set of all its subvarieties satisfies the descending chain condition with respect to inclusion. In 1950 Specht \cite{Specht1950} formulated a problem on the Specht property for the variety of all associative algebras over a field of characteristic zero.  

Specialists extended the study of this problem for any varieties of algebras over fields of any characteristic. 
In 1970 Vaughan-Lee \cite{Vaughan-Lee1970} constructed an example of a finite-dimensional Lie algebra over a field of characteristic  $p=2$ that does not have a finite basis of identities. In 1974 Drensky \cite{Drensky1974} extended this result to fields of any positive characteristic $p>0$.  In 1978  Medvedev \cite{Medvedev1978} showed that varieties of metabelian Malcev, Jordan, alternative, and $(-1,1)$ algebras are Spechtian. In 1984 Umirbaev \cite{Umirbaev1984} proved that the variety of metabelian binary Lie algebras over a field of characteristic $\neq 3$ has the Specht property. In 1980 Medvedev \cite{Medvedev1980} also constructed an example of a variety of solvable alternative algebras over a
field of characteristic $2$ with an infinite basis of identities. 
  In 1985 Umirbaev \cite{Umirbaev1985} proved that the varieties of solvable alternative algebras over a field of characteristic $\neq 2, 3$ have the Specht property. Pchelintsev \cite{Pchelintsev2000} constructed an almost Spechtian variety of alternative algebras over a field of characteristic 3. The Specht property of so-called bicommutative algebras is proven in \cite{Drensky-Zhakhaev2018}. 
  
In 1976 Belkin \cite{Belkin1976} proved that the variety of metabelian right-alternative algebras does not have the Specht property. In 1978 L'vov \cite{Lvov1978} constructed a six-dimensional nonassociative algebra over an arbitrary field satisfying the identity $x(yz)=0$ with an infinite basis of identities. 
In 1986 Isaev \cite{Isaev1986} adapted L'vov's methods for right-alternative algebras and constructed a finite-dimensional metabelian right-alternative algebra over an arbitrary field with an infinite basis of identities.  In 2008 Kuz'min \cite{Kuz'min2008} gave a sufficient condition for the varieties of metabelian right-alternative algebras over a field of characteristic $\neq 2$ to be Spechtian. 

In 1988 Kemer \cite{Kemer1988, Kemer1991} positively solved the famous Specht problem \cite{Specht1950} and proved that every variety of associative algebras over a field of characteristic zero has a finite basis of identities. Later the Specht problem was negatively solved for the variety of associative algebras over fields of positive characteristic $p>0$ \cite{Belov2000,Grishin1999,Shchigolev1999}. It is also known that the varieties of Lie algebras generated by a finite-dimensional algebra over a field of characteristic zero have the Specht property  \cite{Iltyakov1992,KSh1988}. 
Despite the efforts of many specialists in this field, the question of whether the variety of Lie algebras over a field of characteristic zero has the Specht property remains open.

This paper is devoted to the study of the Specht property for the variety of right-symmetric algebras. Recall that an algebra $A$ over a field $\mathbb{F}$ is called {\em right-symmetric} if it satisfies the  identity
\begin{equation}\label{rsym}(a,b,c)=(a,c,b), 
\end{equation}
where $(a,b,c)=(a b) c-a (b c)$ is the associator of $a,b,c\in A$. 

Right-symmetric algebras are Lie admissible, that is, any right-symmetric algebra with respect to the commutator $[x,y]=xy-yx$ is a Lie algebra.  Very often right-symmetric (or left-symmetric) algebras are called pre-Lie algebras and play an important role in the theory of operads \cite{Loday}. Right-symmetric algebras arise in many different areas of mathematics and physics \cite{Burde2006}.

In 1994 Segal \cite{Segal1994} constructed a basis of free right-symmetric algebras. Chapoton and Livernet \cite{ChapotonLivernet2001} and, independently,  L\"ofwall and Dzhumadil'daev \cite{Dzhum-Lof2002} gave other bases of free right-symmetric algebras in terms of rooted trees. The identities of right-symmetric algebras were
studied by Filippov \cite{Fil01}, and he proved that any right-nil
right-symmetric algebra over a field of characteristic zero is right 
nilpotent.  An analogue of the PBW basis Theorem for the universal
(multiplicative) enveloping algebra of a right-symmetric algebra was given in
\cite{KU2004}. The Freiheitssatz and the decidability of the word
problem for one-relator right-symmetric algebras were proven in \cite{KMU2008}.
 Recently, Dotsenko and Umirbaev \cite{DotsenkoUmirbaev2023} determined that the variety of right-symmetric algebras over a field of characteristic zero is Nielson-Schreier, that is, every subalgebra of a free right-symmetric algebra is free.    

   A right-symmetric algebra with an additional identity 
$$
a (b c)=b (a c)
$$ is called a Novikov algebra. The class of Novikov algebras is an important and well-studied subclass of right-symmetric algebras. Recently there was great progress in the study of identities, solvability, and nilpotency \cite{Zel,Fil01,DzT,ShZh,UZh21,TUZ21}. In 2022 Dotsenko, Ismailov, and Umirbaev \cite{DIU2023} proved that (a) every Novikov algebra satisfying a nontrivial polynomial identity over a field of characteristic zero is right-associator nilpotent and (b) the variety of Novikov algebras over a field of characteristic zero has the Specht property.

In this paper, we continue the study of the identities of right-symmetric algebras. Namely, using the constructions and methods of L'vov \cite{Lvov1978} and Isaev \cite{Isaev1986}, we construct a finite-dimensional metabelian right-symmetric algebra over an arbitrary field that does not have a finite basis of identities. In fact, our algebra belongs to the variety of algebras $\mathcal{R}$ defined by the identities 

\begin{equation}\label{def0} [[a,b],c]=0,
\end{equation}
\begin{equation}\label{def1} (a b) a=0,
\end{equation}
and
\begin{equation}\label{def2}(a b) (c d)=0.
\end{equation}

We determine some identities and operator identities of the variety $\mathcal{R}$ in Section 2.  In Section 3, a series of algebras $P_n$ of this variety is constructed. A linear basis of free algebras of the variety $\mathcal{R}$ is constructed in Section 4. Section 5 is devoted to the study of the relationships between the polynomial identities and the operator identities of the algebras $P_n$. The main result of the paper is given in Section 6 and says that the algebra $P_2$ does not have a finite basis of identities.

\section{\label{2}\ A variety of right-symmetric algebras}

Let $\mathbb{F}$ be an arbitrary fixed field. In what follows, all vector spaces are considered over $\mathbb{F}$. As above, $\mathcal{R}$ denotes the variety of algebras defined by the identities (\ref{def0}), (\ref{def1}), and (\ref{def2}).

\begin{lemma}\label{u1} Every algebra of the variety $\mathcal{R}$ is right-symmetric and right nilpotent of index $4$.  
\end{lemma}
\begin{proof} The linearization of (\ref{def1}) gives 
\begin{equation}\label{multi}(a b) c+(c b) a=0.
\end{equation}
This identity and (\ref{def2}) imply that 
\begin{equation}\label{def3}((a b) c) d=-(d c) (a b)=0. 
\end{equation}
 Using (\ref{def1}) and(\ref{def0}) one can also get 
$$(a,b,c)-(a,c,b)=(ab)c-a(bc)-(ac)b+a(cb)$$
$$=-(cb)a-a(bc)+(bc)a+a(cb)$$
$$=[b,c]a-a[b,c]=[[b,c],a]=0,$$
i.e., $\mathcal{R}$ is a variety of right-symmetric  algebras. 
\end{proof}

Let $A$ be an arbitrary algebra of the variety $\mathcal{R}$. Recall that for any $x\in A$ the operators of right multiplication $R_x$ and left multiplication $L_x$ on  $A$ are defined by 
\[aR_x=ax \quad \text{and} \quad  aL_x=xa,\] 
respectively.   Set also $V_{x,y}=L_{x}R_{y}$.  
\begin{lemma}\label{relations}
\begin{equation}\label{rel1}
V_{x,x}=0, \quad \quad  V_{x,y}=-V_{y,x}.
\end{equation}
\begin{equation}\label{rel2}
xR_{y}L_{z}L_{t}=yV_{x,z}L_{t}-xR_{y}V_{t,z}.
\end{equation}
\begin{equation}\label{rel3}
xR_{y}L_{z}=xV_{z,y}+yR_{x}L_{z}-yV_{z,x}.
\end{equation}
\begin{equation}\label{rel4}
xR_{y}V_{z,t}=yR_{x}V_{z,t}.
\end{equation}
\begin{equation}\label{rel5}
V_{x,y}R_z=0.
\end{equation}
\begin{equation}\label{rel6}
V_{x,y}(L_zL_t+V_{t,z})=0.
\end{equation}
\end{lemma}
\begin{proof} The identities (\ref{def1}) and (\ref{multi}) immediately imply  (\ref{rel1}). 
By (\ref{rsym}) and (\ref{def2}) we get
$$xR_{y}L_{z}L_{t}=t(z(xy))$$
$$=(tz)(xy)+t((xy)z)-(t(xy))z=yV_{x,z}L_{t}-xR_{y}V_{t,z}.$$
From the identity (\ref{def0}) follows (\ref{rel3}).

Then (\ref{rsym}) and (\ref{def3}) give that 
$$xR_{y}V_{z,t}=(z(xy))t$$
$$=((zx)y)t+(z(yx))t-((zy)x)t=yR_{x}V_{z,t}.$$
By (\ref{def3}) we obtain $tV_{x,z}R_t=((xt)z)t=0$, and, therefore,  $V_{x,y}R_z=0$.
Set $v=uV_{x,y}$. Then (\ref{def3}), (\ref{rsym}), and  (\ref{def2}) imply that 
$$vL_zL_t=t(zv)-t(vz)=(tz)v-(tv)z=-vV_{t,z}.$$
\end{proof}

\section{\label{2}\ Algebras $P_n$}

For each natural $n$ we define the algebra $P_n$ with a linear basis 
\[a_{ij}, \, b_{ij}, \,c_i, \, d_{ij}, \, e_{ij},\]
where $i,j\in\{1,2,\ldots,n\}$, and with the product defined by  
\[a_{ij}c_i=d_{ij}, \quad b_{ij}c_i=e_{ij},\]
\[a_{ij}e_{ij}=e_{ij}a_{ij}=-b_{ij}d_{ij}=-d_{ij}b_{ij}=c_j,\]
where all zero products are omitted.

Set 
$$A_n=\mathrm{Span}\{a_{ij},\, b_{ij} \, |\,  1\leq i,j\leq n\}$$ 
and 
$$D_n=\mathrm{Span}\{c_i,\, d_{ij},\, e_{ij} \, | \, 1\leq i,j \leq n\},$$
where $\mathrm{Span}\, X$ denotes the linear span of $X$. 
Then $A_n$ is a subalgebra of $P_n$ and $D_n$ is an ideal of $P_n$.  Moreover, $P_n$ is a direct sum of the vector spaces $A_n$ and $D_n$. 
Set also 
$$C_n=\mathrm{Span}\{c_i \,| \,1\leq i \leq n\}, \ \ \overline{C}_n=\mathrm{Span}\{d_{ij},\, e_{ij} \,|\, 1\leq i,j\leq n\}. $$ 

Then 
\begin{equation}\label{main reltions}
P_n^2=D_n, \quad A_n^2=D_n^2=0, \quad D_n=C_n\oplus\overline{C}_n,
\end{equation}
\[D_nP_n=C_n, \quad P_nA_n=C_n,\quad P_nC_n=\overline{C}_n, \quad 
C_nP_n=0,\quad  P_n\overline{C}_n=C_n.\]

\begin{lemma}\label{lemma1} The algebra $P_n$ belongs to the variety $\mathcal{R}$.
\end{lemma}
\begin{proof} Obviously the space of commutators $[P_n,P_n]$ coincides with  $\overline{C}_n$, which is in the center of $P_n$, i.e., (\ref{def0}) holds.

In order to verify the identity (\ref{def1}), it is sufficient to check the identities  (\ref{def1}) and (\ref{multi}) for all elements of the basis of $P_n$. 
Let us begin with (\ref{def1}). Since $A_n^2=D_n^2=(D_nA_n)D_n=0$, we may assume that $a\in A_n$ and $b\in D_n$. Consider all nonzero products of the space $A_nD_n$. If $a=a_{ij}$ and $b=c_i$, then
$$(a_{ij}c_i)a_{ij}=d_{ij}a_{ij}=0.$$ If $a=a_{ij}$ and $b=e_{ij}$, then
$$(a_{ij}e_{ij})a_{ij}=c_{j}a_{ij}=0.$$
The other cases can be verified similarly. 

Now let's verify (\ref{multi}).
Since $(D_nP_n)P_n=0$, the product $(ab)c$ is nonzero only if $a=a_{ij},\, b=c_i, \,c=b_{ij}$ or  $a=b_{ij}, \,b=c_i, \,c=a_{ij}$. Thus,
\[(ab)c+(cb)a=-c_j+c_j=0.\]
 The identity (\ref{def2}) follows from the relations $P^2_n=D_n$ and $D^2_n=0$.   
\end{proof}

\begin{lemma}\label{V-action} For all  $x,y\in P_n$, $d\in D_n$ we have
$$(A_n+\overline{C}_n)V_{x,y}=0, \quad V_{d,y}=V_{y,d}=0. $$
\end{lemma}
\begin{proof} The relations (\ref{main reltions}) give that $(P_nA_n)P_n\subseteq C_n P_n=0$ and $(P_n\overline{C}_n)P_n\subseteq C_n P_n=0$, i.e., the first equality of the lemma holds. Similarly, by noting that $(D_nP_n)P_n\subseteq C_n P_n=0$ and $(P_nP_n)D_n\subseteq D_n D_n=0$, we can deduce the second equality of the lemma. 
\end{proof}

Denote by $\mathrm{Ann}_lP_n$ the space of left annihilators of $P_n$.

\begin{lemma}\label{lemma30} $\mathrm{Ann}_lP_n=C_n$.  
\end{lemma}
\begin{proof}
Assume that $x\in (A_n+\overline{C}_n+C_n)\cap \mathrm{Ann}_lP_n$ and express it as 
$$x=\sum_{i,j}(\alpha_{ij}a_{ij}+\beta_{ij}b_{ij}+\gamma_{ij}d_{ij}+\delta_{ij}e_{ij}+\epsilon_ic_i),$$
where $\alpha_{ij}, \beta_{ij}, \gamma_{ij}, \delta_{ij}, \epsilon_i\in\mathbb{F}$. Then we have
$$xc_i=\sum_{j}(\alpha_{ij}d_{ij}+\beta_{ij}e_{ij}), \quad xa_{ij}=\delta_{ij}c_j, \quad xb_{ij}=-\gamma_{ij}c_j.$$ From these equations, it can be deduced that  $\alpha_{ij}=\beta_{ij}=\gamma_{ij}=\delta_{ij}=0$. Therefore, we can conclude that  $x=\sum_i\epsilon_ic_i$.  Consequently, $\mathrm{Ann}_lP_n=C_n$. 
\end{proof}

\section{\label{1}\ Structure of free algebras of  $\mathcal{R}$}

Let $F(X)$ be the free algebra of the variety $\mathcal{R}$ generated by an infinite countable set $X=\{x_1,x_2,\ldots, x_n,\ldots\}$. 
\begin{prp}\label{basis}
The set of elements $\mathcal{B}$  of $F(X)$ of the forms
\begin{center}
$x_i$, \quad $x_i\Hat{R}_{x_j}L_{x_s}$,  \quad  $x_i\Hat{R}_{x_j}V_{x_{p_1},x_{q_1}}\cdots V_{x_{p_k},x_{q_k}}\Hat{L}_{x_s},$
\end{center}
 where $i<j$ and $p_r<q_r$ for all $r=1,2,\ldots, k$, $k\geq 1$, and $\Hat{T}_x$ denotes that the operator $T_x$ might not occur, is a basis of $F(X)$.     
\end{prp}
\begin{proof} In order to show that $\mathcal{B}$ linearly spans $F(X)$ it is sufficient to verify that, for any $v\in \mathcal{B}$, the elements $v R_{x_i}$ and $vL_{x_i}$ belong to the linear span of 
$\mathcal{B}$. This is easy to do using the identities (\ref{rsym}), (\ref{def2}), and Lemma \ref{relations}. For example, let  
$$v=x_i\Hat{R}_{x_j}V_{x_{p_1},x_{q_1}}\cdots V_{x_{p_k},x_{q_k}}L_{x_s}.$$
Then 
$$vR_{x_r}=x_i\Hat{R}_{x_j}V_{x_{p_1},x_{q_1}}\cdots V_{x_{p_k},x_{q_k}}L_{x_s}R_{x_r}=x_i\Hat{R}_{x_j}V_{x_{p_1},x_{q_1}}\cdots V_{x_{p_k},x_{q_k}}V_{x_s,x_r}.$$
By (\ref{rel6}), we get  
$$vL_{x_r}=x_i\Hat{R}_{x_j}V_{x_{p_1},x_{q_1}}\cdots V_{x_{p_k},x_{q_k}}L_{x_s}L_{x_r}=-x_i\Hat{R}_{x_j}V_{x_{p_1},x_{q_1}}\cdots V_{x_{p_k},x_{q_k}}V_{x_r,x_s}.$$
 Applying (\ref{rel1}) we can express $vR_{x_r}$ and $vL_{x_r}$ as a linear combination of elements of $\mathcal{B}$. 
 
 It remains to prove the linear independence of elements of $\mathcal{B}$. Suppose that $f=f(x_1,x_2,\ldots,x_n)\in F(X)$ is a nontrivial linear combination of elements of $\mathcal{B}$. 
Suppose that  $v\in \mathcal{B}$ and $\deg_{x_i}(v)=k$. Let's write $v=v(x_i,\ldots,x_i)$ in order to differ the presence of $x_i$ in different places. To linearize $v$ in $x_i$ we use new variables $y_1,\ldots,y_k\in X$ and, after renumeration, we can assume that $y_r<x_j$ if $i<j$ and $x_j<y_r$ if $j<i$ for all $1\leq r\leq k$. Notice that every word $v(y_{\sigma(1)},\ldots,y_{\sigma(k)})$, where 
$\sigma\in S_k$ and $S_k$ is the symmetric group in $k$ symbols, is an element of  $\mathcal{B}$. Then the full linearization of $v$ in $x_i$ is a linear combination of basis elements $v(y_{\sigma(1)},\ldots,y_{\sigma(k)})$. Therefore, by linearizing a nontrivial element $f$, we obtain a nontrivial element that is a linear combination of multilinear elements from $\mathcal{B}$. Substituting zeroes instead of some variables, if necessary, we can make $f$ linear in each variable. 
 Therefore,  we can assume that  $f$ is a multilinear nontrivial identity in the variables $x_1,\ldots,x_n$. Let
$$f=\sum_{i=1}\alpha_iu_i,$$
where $\alpha_i\in\mathbb{F}$ and $u_i\in\mathcal{B}$. Suppose, for example, that  $$u_1=x_iR_{x_j}V_{x_{p_1},x_{q_1}}\cdots V_{x_{p_k},x_{q_k}}L_{x_s}.$$
Set  $x_i=d_{1,2}$, $x_j=-b_{1,2}$, $x_{p_r}=a_{r+1,r+2}$, $x_{q_r}=-b_{r+1,r+2}$ for all $r=1,2,\ldots k$, $x_s=a_{k+2,k+3}$.  We have $c_iV_{a_{ij},-b_{ij}}=c_j$ for all $i,j$. 
Then the value of $u_1$ under this substitution is $d_{k+2,k+3}$ and the value of any other $u_i$ is $0$. Consequently, the value of $f$ is $\alpha_1d_{k+2,k+3}\neq 0$. 
Thus, $f$ is not an identity for $\mathcal{R}$. 

If $L_{x_s}$ does not appear in $u_1$, then we perform the same substitutions for the variables. If $R_{x_j}$ does not appear in $u_1$, then we simply set $x_i=c_2$ and perform the same substitutions for the rest of the variables as described above. In both cases the value of $f$ is nonzero. This completes our proof. 
\end{proof}

Let $M=M(F(X))$ be the multiplication algebra of the algebra $F(X)$. Denote by $E_0$ the subalgebra (without identity) of $M$ generated by the operators $V_{x_i,x_j}$ with $i<j$ for all $i,j=1,2,\ldots$. Set also 
$$E_1=\sum_{j\geq 1}E_0L_{x_j}, \quad E_2=\sum_{i\geq 1}R_{x_i}E_0, \quad E_3=\sum_{i,j\geq 1}R_{x_i}E_0L_{x_j},$$
and  
$$R_k=\sum_{i\geq 1}x_iE_k, \quad \mbox{for } k=0,1,2,3.$$

According to Proposition \ref{basis}, the space $F(X)$ is the direct sum of the subspaces $R_k$ and the linear span of the elements of $\mathcal{B}$ of degree less than or equal to 3. 

\begin{lemma}\label{lemma6}
An identity $zf(x_1,\ldots,x_m)=0$, where $f\in E_0$, is a consequence of a system of identities 
\begin{equation}\label{a system of equations}
tg_j(x_1,\ldots,x_l)=0, \quad g_j\in E_0, \quad j\in \mathcal{J}, 
\end{equation}
in the variety $\mathcal{R}$, where $\mathcal{J}$ is any set of indices, if and only if the operator $f(x_1,\ldots,x_m)$ belongs to the ideal of the associative algebra $E_0$ generated by the set $G$ of all operators $\varphi(g_j)$, where $\varphi$ runs over the set of all linear endomorphisms   $\varphi:X\rightarrow \mathbb{F}X=\sum_{i\geq 1}\mathbb{F}x_i$ and $j\in \mathcal{J}$.
\end{lemma}
\begin{proof}
Suppose that $f$ belongs to the ideal of $E_0$ generated by $G$. Then $$f=\sum_{r=1}^{t}u_rg_{j_r}^{\varphi_r}v_r,$$
for some linear endomorphisms $\varphi_r$ and $u_r,v_r\in E_0$. Therefore, 
$$zf=\sum_{r=1}^{t}(zu_r)g_{j_r}^{\varphi_r}v_r$$
and $zf=0$ is a consequence of the system of identities (\ref{a system of equations}).

Let's describe all the consequences of the identities (\ref{a system of equations}). Let $\varphi: F(X)\to F(X)$ be an arbitrary endomorphism and set $\varphi(x_i)=y_i+h_i$, where $y_i\in \mathbb{F}X$ and $h_i\in F(X)^2$ for all $i$. Since $g_j\in E_0$, using (\ref{def2}) and (\ref{def3}), we get
$$t\varphi(g_j)=tg_j(y_1,\ldots,y_l)=0.$$

Thus, a general form of consequences of the identities   (\ref{a system of equations}) can be expressed as
$$\sum_{r=1}^{t}u_rg_{j_r}^{\varphi_r}v_r,$$ 
where $u_r\in F(X)$, $v_r\in M(F(X))$, and $\varphi_r$ are linear endomorphisms. We know that 
$g_{j_r}^{\varphi_r}\in E_0$. 
We also claim that $u_r$ and $v_r$ can be represented in the forms
$$x_i\Hat{R}_{x_j}V_{x_{p_1},x_{q_1}}\cdots V_{x_{p_k},x_{q_k}}, \quad\quad V_{x_{p'_1},x_{q'_1}}\cdots, V_{x_{p'_k},x_{q'_k}}\Hat{L}_{x_s},$$
respectively, 
where $i<j$, $p_l<q_l$,  $p'_l<q'_l$ and $k=0,1,\ldots$.

Suppose that $u_r$ is a basis element that ends with  $L_{x_s}$.  Then, by (\ref{rel5}) and  (\ref{rel6}), we can derive that 
$$V_{x_i,x_j}L_{x_s}V_{x_k,x_t}=V_{x_i,x_j}L_{x_s}L_{x_k}R_{x_t}=-V_{x_i,x_j}V_{x_k,x_s}R_{x_t}=0.$$
Consequently, we have $V_{x_i,x_j}L_{x_s}E_{0}=0$.

If  $u_r=x_iR_{x_j}L_{x_k}$,  then by (\ref{rel2}) and (\ref{rel5})  we get
$$u_rV_{x_s,x_t}=x_iR_{x_j}L_{x_k}L_{x_s}R_{x_t}=x_jV_{x_i,x_k}L_{x_s}R_{x_t}-x_iR_{x_j}V_{x_s,x_k}R_{x_t}=x_jV_{x_i,x_k}V_{x_s,x_t}.$$
So, we can conclude that $u_r$ has the claimed form.

Now, let's consider the case when $v_r$ is a basis element that starts with $R_y$. According to (\ref{rel5}), we have $E_0R_y=0$. If $v_r$ starts with $L_{x_i}L_{x_j}$, then by using (\ref{rel6}), we find
$$V_{x_k,x_s}L_{x_i}L_{x_j}=-V_{x_k,x_s}V_{x_j,x_i}.$$
Hence,  $v_r$ also has the claimed form.

If $zf=0$ is a consequence of the identities  (\ref{a system of equations}), then we get an equality of the form 
$$x_{m+1}f(x_1,\ldots,x_m)=\sum_{r=1}^{t}\lambda_rx_{i_r}w_rg_{j_r}^{\varphi_r}v_r,$$
where $x_{i_r}w_r=u_r$, $w_r\in E_0+E_2$ and $v_{r}\in E_0+E_1$. Notice that every element 
$x_{i_r}w_rg_{j_r}^{\varphi_r}v_r$ belongs to $\mathcal{B}$. Consequently, we may assume that 
$x_{i_r}=x_{m+1}$, $w_r, v_{r}\in E_0$,  and 
$$f(x_1,\ldots,x_m)=\sum_{r=1}^{t}\lambda_rw_rg_{j_r}^{\varphi_r}v_r.$$ 
\end{proof}

\section{\label{3}\ Identities of $P_n$.}

In this section, we study the connections between the identities and the operator identities of  $P_n$ for $n\geq 2$.

\begin{lemma}\label{lemma2}
If $f=f(x_1,\ldots,x_m)\in F(X)$ and $f=0$ is an identity of $P_n$ for $n\geq 2$, then 
\begin{equation}\label{decomp}
f=f_0+f_1+f_2+f_3\in F(X), \quad  f_k\in R_k, 
\end{equation}
 and  $f_k=0$ is an identity of $P_n$ for all $k=0,1,2,3$.
\end{lemma}
\begin{proof} Let 
$$f=\sum_{i=1}^m\lambda_ix_i+\sum_{i,j=1}^m\lambda_{ij}x_ix_j+\sum_{i,j,k=1,i<j}^m\lambda_{ijk}x_iR_{x_j}L_{x_k}+f',$$
where $f'$ is a linear combination of elements from $\mathcal{B}$ of degree $\geq 4$. 

We first show that $\lambda_i=\lambda_{ij}=\lambda_{ijk}=0$ for all $i,j,k=1,\ldots,m$. For any  fixed $i$ the substitution $x_i=c_1$ and $x_j=0$ for all $j\neq i$ gives that $\lambda_i c_1=0$, which implies $\lambda_i=0$. 

If $i\neq j$, then the substitution $x_i=a_{11}$, $x_j=c_1$,  and $x_k=0$ for all $k\neq i,j$, makes the value of $f$ equal to $\lambda_{ij}d_{11}=0$. We get the same value if $i=j$ under the substitution $x_i=x_j=a_{11}+c_1$ and $x_k=0$ for all $k\neq i,j$. This gives $\lambda_{ij}=0$ in both cases. 

Assume that $i<j>k$. If $i\neq k$, then the substitution $x_i=b_{11}$, $x_j=d_{11}$, $x_k=a_{12}$, and $x_t=0$ for all $t\neq i,j,k$, makes the value of $f$ equal to $-\lambda_{ijk}d_{12}$. This gives that $\lambda_{ijk}=0$. If $i=k$, then the substitution $x_i=b_{11}$, $x_j=d_{11}$, and $x_t=0$ for all $t\neq i,j$, gives that $-\lambda_{iji}e_{11}=0$ and $\lambda_{iji}=0$.  If  $i<j=k$, then the substitution $x_i=d_{11}$, $x_j=b_{11}$, and $x_t=0$ for all $t\neq i,j$, gives that $-\lambda_{ijj}e_{11}=0$ and $\lambda_{ijj}=0$. Finally, if $i<j<k$, then the  substitution $x_i=d_{11}$, $x_j=x_k=b_{11}$, and $x_t=0$ for all $t\neq i,j$, gives that $-\lambda_{ijk}e_{11}-\lambda_{ikj}e_{11}=0$, i.e., $\lambda_{ijk}=-\lambda_{ikj}=0$.  

Thus, $f$ is a linear combination of elements of $\mathcal{B}$ of degree $\geq 4$.  Suppose that $f$ is written as in (\ref{decomp}). Taking into account the relations $D_nP_n\subseteq C_n$ and  $P_nC_n\subseteq\overline{C}_n$ it can be observed that the images of $F_0=f_0+f_2$ and $F_1=f_1+f_3$ belong to $C_n$ and  $\overline{C}_n$, respectively.   Therefore, if $f=0$ is an identity of $P_n$, then  $F_0=0$ and $F_1=0$ are also identities of $P_n$.

Suppose that $$f_k(x_1,\ldots,x_m)=\sum_{i=1}^mx_ig_i^{(k)}(x_1,\ldots,x_m),$$ where $g_i^{(k)}\in E_{k}$ and $k=0,1,2,3$.
Let $p_1,\ldots,p_m\in P_n$ with $p_s=v_s+\overline{v}_s+a_s$, where $v_s\in C_n$, $\overline{v}_s\in \overline{C}_n$, $a_s\in A_n$. By Lemma \ref{V-action} we can obtain that
$$p_iV_{p_j,p_k}=v_iV_{p_j,p_k}=v_iV_{v_j+\overline{v}_j+a_j,v_k+\overline{v}_k}+v_iV_{v_j+\overline{v}_j,a_k}+v_iV_{a_j,a_k}=v_iV_{a_j,a_k}.$$
Then 
$$f_0(p_1,\ldots,p_m)=
\sum_{i=1}^mv_ig_i^{(0)}(a_1,\ldots,a_m)=f_0(v_1+a_1,\ldots,v_m+a_m).$$
It is easy to see that $(A_n+C_n)R_{a_i+v_i}V_{a_j+v_j,a_s+v_s}=0$ for all $a_i,a_j,a_s\in A_n$ and $v_i,v_j,v_s\in C_n$. It follows that
$$f_2(v_1+a_1,\ldots,v_m+a_m)=0.$$
Thus, $$f_0(p_1,\ldots,p_m)$$
$$=f_0(v_1+a_1,\ldots,v_m+a_m)+f_2(v_1+a_1,\ldots,v_m+a_m)$$$$=F_{0}(v_1+a_1,\ldots,v_m+a_m)=0.$$

Therefore, we can conclude that $f_0=0$ and $f_2=0$ are identities of $P_n$. Similarly, we can establish that $f_1=0$ and $f_3=0$ are also identities of $P_n$.
\end{proof}

\begin{lemma}\label{lemma3} 
If $f=f(x_1,\ldots,x_m)\in R_1+R_3$, then $fx_{m+1}\in R_0+R_2$ and if \\$f(x_1,\ldots,x_m)x_{m+1}=0$ is an identity of $P_n$, then $f=0$ is an identity of $P_n$ as well. 
\end{lemma}
\begin{proof} We have $fx_{m+1}\in R_0+R_2$ by the definition of the spaces $R_i$, where $0\leq i\leq 3$. If $fx_{m+1}=0$ is an identity of $P_n$, then all values of $f$ in $P_n$ belong to 
$C_n=\mathrm{Ann}_l(P_n)$ by Lemma \ref{lemma30}. However, since $f$ is an element of $R_1+R_3$, the values of $f$ must belong to $\overline{C}_n$. Consequently, $f=0$ is an identity of $P_n$.
\end{proof}

Recall an exact formal definition of the linearization of identities \cite[Chapter 1]{ZSSS1982}.  Let $\mathcal{V}$ be an arbitrary variety of algebras and $\mathbb{F}\langle X\rangle$ be its free algebra over $\mathbb{F}$ generated by $X=\{x_1,x_2,\ldots\}$. Let  $y\in\mathbb{F}\langle X\rangle$ be an arbitrary fixed element. For a nonnegative integer $k$, we define the linear mapping $\Delta^k_{x_i}(y)$ on $\mathbb{F}\langle X\rangle$ as follows:
\begin{itemize}
\item $\Delta^0_{x_i}(y)$ is the identity mapping;
\item $x_s\Delta^k_{x_i}(y)=0$, if either $k>1$ or $k=1$, $i\neq s$;
\item $x_i\Delta^1_{x_i}(y)=y$;
\item $(uv)\Delta^k_{x_i}(y)=\sum_{r+s=k}(u\Delta^r_{x_i}(y))(v\Delta^s_{x_i}(y))$,
\end{itemize}
where $x_i\in X$ and $u,v$ are any monomials in $\mathbb{F}\langle X\rangle$. We also write $\Delta_{x_i}(y)$ instead of $\Delta^1_{x_i}(y)$. 

\begin{lemma}\label{lemma4}
Suppose that  $f=f(x_1,\ldots,x_m)\in R_2$. Then $f\Delta_i(x_{m+1}x_{m+2})\in R_0$ for all $1\leq i\leq m$. Moreover, $f=0$ is an identity of $P_n$ if and only if $P_n$ satisfies the following system of identities 
\begin{equation}\label{partial lin}
f(x_1,\ldots,x_m)\Delta_i(x_{m+1}x_{m+2})=0, \quad 1\leq i\leq m.
\end{equation} 
\end{lemma}
\begin{proof} 
Let $w=xR_yV_{z_1,t_1}\cdots V_{z_r,t_r}\in \mathcal{B}$ and $u,v\in X$. We have 
$$(xR_y)\Delta_x(uv)=(uv)R_y=vV_{u,y}.$$
By (\ref{rsym}), (\ref{def2}) and (\ref{def3}), we get 
$$(xR_yV_{z_1,t_1})\Delta_y(uv)=(z_1(x(uv)))t_1=((z_1x)(uv)-(z_1(uv))x+z_1((uv)x))t_1$$
$$=-((z_1(uv))x)t_1+(z_1((uv)x))t_1=(z_1((uv)x))t_1=vV_{u,x}V_{z_1,t_1}.$$
By (\ref{def2}) we get that $w\Delta_{z_i}(uv)=w\Delta_{t_i}(uv)=0$ for any $i=1,\ldots,r$.    
Thus, if $f(x_1,\ldots,x_m)\in R_2$, then  $f(x_1,\ldots,x_m)\Delta_i(x_{m+1}x_{m+2})\in R_0$.

 If $p_1,\ldots,p_m\in P_n$ and $v_1,\ldots,v_m\in D_n$, then we have 
\begin{equation}\label{helpful relation}
f(p_1+v_1,\ldots,p_m+v_m)=f(p_1,\ldots,p_m)+\sum_{i=1}^{m}f(p_1,\ldots,p_m)\Delta_i(v_i).
\end{equation}
In fact, by Lemma 1.3 from \cite{ZSSS1982},  the relation
$$f(x_1+y_1,\ldots,x_m+y_m)=\sum_{i_1,\ldots,i_m\geq 0}f\Delta_1^{i_1}(y_1)\cdots\Delta_m^{i_m}(y_m)$$
$$=f(x_1,\ldots,x_m)+\sum_{i=1}^{m}f(x_1,\ldots,x_m)\Delta_i(y_i)+g,$$
where $y_1,\ldots,y_m\notin\{x_1,\ldots,x_m\}$ are distinct variables and the degree of $g$ in the variables $y_1,\ldots,y_m$ is greater than one, holds in $\mathbb{F}\langle X\rangle$. By substituting  $x_i=p_i, y_i=v_i$ and using the fact that $D_n^2=0$, we can obtain the relation (\ref{helpful relation}).

If $f=0$ is an identity of $P_n$, then the relation (\ref{helpful relation})  implies that  
$$f(p_1,\ldots,p_m)\Delta_i(v)=f(p_1,\ldots,p_i+v,\ldots,p_m)-f(p_1,\ldots,p_m)=0$$ for all $p_i\in P_n$ and $v\in D_n$. In other words, the algebra $P_n$ satisfies the system of identities (\ref{partial lin}).

Conversely, suppose that the system of identities (\ref{partial lin}) holds in $P_n$. Assume that $p_1,\ldots,p_m\in P_n$ of the form  $p_i=a_i+v_i$, where $a_i\in A_n$ and $v_i\in D_n$. Then using the relation (\ref{helpful relation}), we have 
$$f(p_1,\ldots,p_m)=f(a_1+v_1,\ldots,a_m+v_m)$$
$$=f(a_1,\ldots,a_m)+\sum_{i=1}^{m}f(p_1,\ldots,p_m)\Delta_i(v_i)=f(a_1,\ldots,a_m).$$ 
Considering $A^2_n=0$ and $f\in R_2\subseteq F(X)^2$, we can conclude that $f(a_1,\ldots,a_m)=0$. 
Consequently, $f(p_1,\ldots,p_m)=0$. 
 \end{proof}

\begin{lemma}\label{lemma5}
If $f=f(x_1,\ldots,x_m)\in R_0$ and $f=0$ is an identity of $P_n$ of the form
$$f=\sum_{i=1}^{m}x_ig_i,$$
where $g_i\in E_0$, then $x_{m+1}g_i=0$ is an identity of $P_n$.
\end{lemma}
\begin{proof}
For a fixed $i$ set $x_i=v+a_i$ and $x_j=a_j$ for all $j\neq i$, where $v\in D_n$ and $a_j\in A_n$. Taking into account the relations $A_n^2=D_n^2=0$ and Lemma \ref{V-action}, one can have
$$f(x_1,\ldots,x_m)=vg_i(a_1,\ldots,a_m)=0.$$ Hence, $x_{m+1}g_i=0$ is an identity of $P_n$.
\end{proof}

\begin{prp}\label{cor} For an arbitrary polynomial $f=f(x_1,\ldots,x_m)\in F(X)$ there exist $t(m)=2m(m+3)$ polynomials   $g_i(x_1,\ldots,x_{m+3})\in E_0$, where $i=1,\ldots,t(m)$, such that $f(x_1,\ldots,x_m)=0$ is an identity of $P_n$  for $n\geq 2$ if and only if $P_n$ satisfies the system of identities 
$$zg_i(x_1,\ldots,x_{m+3})=0, \quad  1\leq i\leq t(m).$$   
\end{prp}
\begin{proof}
Let  $f=f(x_1,\ldots,x_m)\in F(X)$ and  suppose that $f=0$ is an identity of $P_n$. Then by Lemma \ref{lemma2} we obtain
$$f=f_0+f_1+f_2+f_3, \quad f_k\in R_k,$$  and $f_k=0$ is an identity of the algebra $P_n$. 

 By Lemma \ref{lemma5}, the identity $f_0=\sum_{i=1}^mx_ig_{i}=0$ is equivalent to the system of $m$ identities $x_{m+1}g_i=0$ of $P_n$, where $1\leq i\leq m$.  

By Lemma  \ref{lemma3}, the identity $f_1=0$ is equivalent to $f_1x_{m+1}=0$ and $f_1x_{m+1}\in R_0$. Moreover, if 
$$f_1x_{m+1}=\sum_{i=1}^mx_ig_i, \quad g_i\in E_0,$$
then, by Lemma \ref{lemma5}, the identity $f_1x_{m+1}=0$ is equivalent to the system of $m$ identities $x_{m+2}g_i=0$ of $P_n$, where $1\leq i\leq m$.  

By Lemma \ref{lemma4}, the identity $f_2=0$ is equivalent to the system of $m$ identities \\$f_2(x_1,\ldots,x_m)\Delta_i(x_{m+1}x_{m+2})=0$, where $i=1,\ldots,m$,  and we have  $f_2\Delta_i(x_{m+1}x_{m+2})\in R_0$. Hence, by Lemma \ref{lemma5}, it is equivalent to a system of $m(m+2)$ identities of the form $x_{m+3}g_i=0$, where $g_i(x_1,\ldots,x_{m+2})\in E_0$ and  $i=1,\ldots,m(m+2)$.

By  Lemma  \ref{lemma3}, the identity $f_3=0$ is equivalent to $f_3x_{m+1}=0$ and $f_3x_{m+1}\in R_2$. The identity (\ref{def2}) implies that $(f_3x_{m+1})\Delta_{m+1}(x_{m+2}x_{m+3})=0$. Then, by  Lemma \ref{lemma4}, $f_3=0$ is equivalent to the system of $m$ identities  $0=(f_3x_{m+1})\Delta_i(x_{m+2}x_{m+3})\in R_0$, where $i=1,\ldots,m$. Moreover, by Lemma \ref{lemma5}, it is equivalent to a system of $m(m+2)$ identities of the form $x_{m+4}g_j=0$, where $g_j(x_1,\ldots,x_{m+3})\in E_0$ and $1\leq j\leq m(m+2)$.

Thus, $f=0$ is equivalent to a system of $t(m)=2m(m+3)$ identities of the form $zg_i(x_1,\ldots,x_{m+3})=0$, where $g_i(x_1,\ldots,x_{m+3})\in E_0$ and $i=1,\ldots,t(m)$.
\end{proof}

Let $B$ be an arbitrary algebra in $\mathcal{R}$.  We define $E_{0}(B)$ as the algebra of operators generated by $V_{b_1,b_2}$ for all $b_1,b_2\in B$, that acts on the algebra $B$. Denote by $T(E_0(B))$ the ideal of $E_0$ defined as the intersection of the kernels of all possible homomorphisms from $F(X)$ to $B$. The elements of $T(E_0(B))$ are called {\it $V$-identities} of $B$. 
\begin{lemma}\label{helpful lemma 2}
$E_0(P_n)\cong M_n(\mathbb{F})$, where $M_n(\mathbb{F})$ is algebra of $n\times n$ matrices.
\end{lemma}
\begin{proof}
According to Lemma \ref{V-action}, $E_0(P_n)$ annihilates the subspace $A_n+\overline{C}_n$, and $C_n$ is an invariant subspace of $P_n$ under its action. Consequently, $E_0(P_n)$ is isomorphic to a subalgebra $L$ of the algebra $End_{\mathbb{F}}C_n$. Furthermore, the operator $V_{b_{ij},a_{ij}}\in E_0(P_n)$ sends the element $c_i$ to $c_j$, and $c_k$ to zero if $k\neq i$, resembling the action of a unit matrix. Therefore, the subalgebra $L$ coincides with the entire algebra $End_{\mathbb{F}}(C_n)\cong M_n(\mathbb{F})$.
\end{proof}

\begin{prp}\label{sufficient and necessary conidition}
If the algebra $P_n$ has a finite basis of identities for $n\geq 2$,  then the ideal $T=T(E_0(P_n))$ is generated by polynomials of bounded degrees. 
\end{prp}
\begin{proof}
 Suppose that $P_n$ has a finite basis of identities for  $n\geq 2$. By Proposition \ref{cor},  modulo (\ref{rsym}), (\ref{def1}), and  (\ref{def2}), every identity is equivalent to a finite system of identities of (\ref{a system of equations}).  Consequently, by Lemma \ref{lemma6}, there exists a finite set of elements $G\subseteq T$ such that the identities $tg=0$, where $g\in G$, form a basis of identities of $P_n$. Let $m$ be the maximum of the degrees of polynomials in $G$. By the same Lemma \ref{lemma6},  the ideal $T$  is generated by all $\varphi(g)$, where $g\in G$ and $\varphi$ is linear. Consequently, $T$ is generated by elements of degrees $\leq m$.  
\end{proof}

\section{\label{4}\ Identities of $P_2$.}

We are going to prove that $P_2$ does not have a finite basis of identities. First, let's construct some important examples of algebras.

\begin{prp}\label{lemman}
For any $s>5$ there exists an algebra $B\in \mathcal{R}$ with the following two properties:
\begin{enumerate}
    \item $B$ is generated by a set $Q=\{q_1,\ldots,q_{s+3}\}$ such that $T\nsubseteq T(E_0(B))$. 
    \item Let $C$ be a subalgebra of $B$ generated by any subset  $Q'$ of $Q$ with $s$ elements. Then 
$$tg(c_1,\ldots,c_k)=0$$
        for all $g(x_1,\ldots,x_k)\in T$, $c_1,\ldots,c_k\in C$, and $t\in B$.
\end{enumerate}
\end{prp}
\begin{proof} Set $n=s-5\geq 1$. Let $H$ be the free algebra with identity in the variety of algebras generated by the field $\mathbb{F}$ with free generators $\{h_1,\ldots,h_n\}$. Denote by $W$  the subspace of $H$, spanned by all words in $h_1,\ldots,h_n$, including the identity $1$, that do not contain at least one $h_i$. Then $W\neq H$. By Theorem 1.6 from \cite{ZSSS1982},  the algebra $A=H\otimes_\mathbb{F}P_3$ belongs to $\mathcal{R}$. Consider the subalgebra $L$ of $A$ generated by the following set of elements:  
\begin{equation}\label{subalgebra L}
\{1\otimes c_1, 1\otimes a_{11}, 1\otimes b_{11}, 1\otimes a_{12}, 1\otimes b_{12}, h_i\otimes a_{22},1\otimes b_{22}, 1\otimes a_{23}, 1\otimes b_{23}\}
\end{equation}
where $i=1,\ldots,n$.

We note that 
$$1\otimes c_2=-(1\otimes b_{12})((1\otimes a_{12})(1\otimes c_1)),$$ 
$$h_j\otimes c_2=-(1\otimes b_{22})((h_j\otimes a_{22})(1\otimes c_2)),$$
$$h_ih_j\otimes c_2=-(1\otimes b_{22})((h_i\otimes a_{22})(h_j\otimes c_2)).$$
Thus, by induction on the length of $h$, one can derive that $h\otimes c_2\in L$ for any word $h$ in $h_1,\ldots,h_n$. In addition, $h\otimes c_3\in L$ since
$$h\otimes c_3=-(1\otimes b_{23})((1\otimes a_{23})(h\otimes c_2)).$$ Note that $h\otimes c_3$ is a two-sided annihilator of $L$ since
$$L\subseteq H\otimes (D_3+\sum_{i\leq j\leq 3, (i,j)\neq(3,3)}(\mathbb{F}a_{ij}+\mathbb{F}b_{ij}).$$

Consequently, $N=H\otimes c_3$ and $N'=W\otimes c_3$ are ideals of $L$. Set $B=L/N'$ and let's show that it satisfies the properties (1) and (2) of the proposition.

{\it Verification of Property (1).} Denote by $q_1,\ldots,q_{s+3}$ the images of the generators of (\ref{subalgebra L}) under the natural projection $L\rightarrow L/N'$. By Lemma \ref{helpful lemma 2},  the algebra $E_{0}(P_2)$ satisfies the well-known Hall's identity 
$$[[\overline{f}_1,\overline{f}_2]\circ [\overline{f}_3,\overline{f}_4],\overline{f}_5]=0,$$
for all $\overline{f}_i\in E_{0}(P_2)$, where $1\leq i\leq 5$ and $a\circ b=ab+ba$. It follows that $S=[[f_1,f_2]\circ [f_3,f_4],f_5]\in T$ for all $f_i\in E_0$. 

It is easy to choose $f_1,\ldots,f_5\in E_0$ and $\varphi : F(X)\to L$ such that 
$$f^\varphi_1=V_{1\otimes b_{12},1\otimes a_{12}}\prod_{i=1}^nV_{1\otimes b_{22},h_i\otimes a_{22}}, \quad f^\varphi_2=f^\varphi_5=V_{1\otimes b_{11},1\otimes a_{11}},$$
$$f^\varphi_3=V_{1\otimes b_{22},1\otimes a_{22}}, \quad f^\varphi_4=V_{1\otimes b_{23},1\otimes a_{23}}.$$ 
The actions of the operators $f^\varphi_1, f^\varphi_2, f^\varphi_3, f^\varphi_4$ on $L$ give us 
$$f^\varphi_1f^\varphi_2=0, \quad f^\varphi_2f^\varphi_1=V_{1\otimes b_{12},v\otimes a_{12}},$$
$$f^\varphi_3f^\varphi_4=V_{1\otimes b_{23},1\otimes a_{23}}, \quad  f^\varphi_4f^\varphi_3=0,$$
and we have 
$$S^\varphi=[-V_{1\otimes b_{12},v\otimes a_{12}}\circ V_{1\otimes b_{23},1\otimes a_{23}},V_{1\otimes b_{11},1\otimes a_{11}}]=V_{1\otimes b_{12},v\otimes a_{12}}V_{1\otimes b_{23},1\otimes a_{23}},$$ where $v=h_1\cdots h_n$.
Since $$(1\otimes c_1)S^\varphi=v\otimes c_3\neq 0 (mod\, N'),$$ we obtain  $S\notin T(E_0(B))$ and therefore $T\nsubseteq T(E_0(B))$. 

{\it Verification of Property (2).} Let $L'$ be a subalgebra of $L$ generated by a subset of the set (\ref{subalgebra L}) that contains no more than $s$ elements. Assume that $f(x_1,\ldots,x_k)\in T$. Let $M$ be the set of all elements of the form $(1\otimes c_1)f(l_1,\ldots,l_k)$, where $l_i\in L'$. 
We claim that 
\begin{equation}\label{NN}
M\cap N\subseteq N'.
\end{equation}

Let's assume that (\ref{NN}) does not hold. In other words, there is an element
$$g=(h_1\cdots h_n\otimes c_3)(h'\otimes c_3)+h''\otimes c_3\in M\cap N$$
for some nonzero $h'\in H$ and some $h^{''}\in W$. 

Note that 
$$(1\otimes c_1)V_{1\otimes b_{12},1\otimes a_{12}}=1\otimes c_2,\quad (1\otimes c_2)V_{1\otimes b_{22},h_i\otimes a_{22}}=h_i\otimes c_2,$$
$$(h_i\otimes c_2)V_{1\otimes b_{22},h_j\otimes a_{22}}=h_ih_j\otimes c_2,\, (h_1\cdots h_n\otimes c_2)V_{1\otimes b_{23},1\otimes a_{23}}=h_1\cdots h_n\otimes c_3.$$

Without using $1\otimes c_1$ and all of the operators 
$$ V_{1\otimes b_{12},1\otimes a_{12}},\quad V_{1\otimes b_{22},h_i\otimes a_{22}}, \quad V_{1\otimes b_{23},1\otimes a_{23}},$$
we cannot get elements containing the product $h_1\ldots h_n$. 
This means that $M\cap N$ contains $g$ if and only if the following $s+1$ elements appear in our calculations: 
$$1\otimes c_1, 1\otimes a_{12}, 1\otimes b_{12}, h_i\otimes a_{22},1\otimes b_{22}, 1\otimes a_{23}, 1\otimes b_{23}\quad (i=1,\ldots,n).$$ 
It is impossible since $L'$ is generated by $s$ elements.  This contradiction establishes the inclusion (\ref{NN}).

Set $\overline{f}=f(l_1,\ldots,l_k)$ for some fixed $l_1,\ldots,l_k\in L'$. 
We show that $L\overline{f}=0\,(mod\, N')$. By Lemma \ref{V-action}, one can have that  $(H\otimes(\overline{C}_3+A_3+\mathbb{F}c_3))E_0(L)=0$. Then, because of (\ref{NN}), it is sufficient to prove that
\begin{equation}\label{NN1}
(1\otimes c_1)\overline{f}=0 (mod\,N),\quad (H\otimes c_2)\overline{f}=0.
\end{equation}
If $l_i=l_i'+l_i''$, where 
$$l_i'\in H\otimes\sum _{i\leq j\leq 2}(\mathbb{F}a_{ij}+\mathbb{F}b_{ij}), \quad l_i''\in H\otimes(\mathbb{F}a_{23}+\mathbb{F}b_{23}),$$ 
then 
\begin{equation}\label{NN2}
(1\otimes c_1)\overline{f}=(1\otimes c_1)f(l_1',\ldots,l_k')(mod\, N).
\end{equation}
We have
$$f(x_1+y_1,\ldots,x_k+y_k)=f(x_1,\ldots,x_k)+g(x_1,\ldots,x_k,y_1,\ldots,y_k),$$
where every monomial of $g\in E_0$ involve at least one variable from $y_1,\ldots,y_k$. Since $LV_{l''_i,L}\subseteq N$, we obtain
$$(1\otimes c_1)g(l'_1,\ldots,l'_k,l''_1\ldots,l''_k)=0(mod\,N),$$
which implies (\ref{NN2}).

Consequently, to prove the first relation of (\ref{NN1}), without loss of generality we can assume that $l_i\in H\otimes\sum _{i\leq j\leq 2}(\mathbb{F}a_{ij}+\mathbb{F}b_{ij})$. In this case, the elements $c_1,l_1,\ldots,l_k$ generate a subalgebra of $H\otimes P_2$. The algebra $H$ can be embedded into the Cartesian product $\mathbb{F}^{\alpha}$ of the algebra $\mathbb{F}$. So, $H\otimes P_2$ can be embedded into $\mathbb{F}^{\alpha}\otimes P_2\cong P^{\alpha}_2$ and satisfies all the identities of $P_2$.  Consequently, it satisfies all $V$-identities from $T$. Thus, the first relation of (\ref{NN1}) is established. The second relation of (\ref{NN1}) can be established similarly using the equality
$$H\otimes(\sum_{j\leq 2}(\mathbb{F}a_{1j}+\mathbb{F}b_{1j}))c_2=0.$$
In this case, we can assume that 
$$l_i\in H\otimes(\sum_{2\leq i,\, j\leq 3}(\mathbb{F} a_{ij}+\mathbb{F}b_{ij})).$$
Then the elements $c_2,l_1,\ldots,l_k$ generate an algebra isomorphic to a subalgebra of $H\otimes P_2$. Thus, we have $L\overline{f}=0(mod\, N')$. The factorization by $N'$ completes the proof of Proposition \ref{lemman}. 
\end{proof}

\begin{lemma}\label{lemma7}
Let $\Sigma$ be a set of generators of the ideal $T=T(E_0(P_2))$ of $E_0$. Then for any natural number $s$, there exists a polynomial $f_s\in\Sigma$ that depends on more than $s$ variables.
\end{lemma}
\begin{proof} Suppose, contrary, that $\Sigma$ consists of polynomials that depend on $\leq s$ variables. Let $B$ be the algebra satisfying the conditions of Proposition \ref{lemman}. 
 Consider the epimorphism $\tau:F(X)\rightarrow B$ defined by
$$\tau(x_i) =
 \left\{\begin{array}{rcl}
         q_{i} & \mbox{if}
         & i\leq s+3, \\ 0  & \mbox{if} & i>s+3.
         \end{array}\right.$$
This induces the epimorphism $\Tilde{\tau}:E_0\rightarrow E_0(B)$ defined by  
$$\Tilde{\tau}g(x_1,\ldots,x_k)=g(\tau(x_1),\ldots,\tau(x_k)).$$ 
If $g(x_1,\ldots, x_k)\in\Sigma$, then $k\leq s$ and $c_i=\tau(x_i)$ belong to a subalgebra of $B$ generated by $\leq s$ elements. By Proposition \ref{lemman}(2), $g(c_1,\ldots,c_k)=0$. 
 Thus, $g\in Ker\, \Tilde{\tau}$. So  $Ker\Tilde{\tau}$ contains $\Sigma$ and, consequently,  $T$ as well. 
 
 Now let $f(x_1,\ldots,x_m)\in T$. For any $b_1,\ldots,b_m\in B$ there exist $r_i\in F(X)$ such that $b_i=\tau(r_i)$ for all $i=1,\ldots,m$. Since $f(r_1,\ldots,r_m)\in T$, we have $f(b_1,\ldots,b_m)=\Tilde{\tau}f(r_1,\ldots,r_m)=0$.  This proves that every element of $T$ is a $V$-identity of $B$. This contradicts Proposition \ref{lemman}(1).
\end{proof}

\begin{theorem}\label{main theorem}
Algebra $P_2$ over an arbitrary field $\mathbb{F}$ does not have a finite basis of identities.
\end{theorem}
\begin{proof} If $P_2$ has a finite basis of identities, then the ideal $T=T(E_0(P_2))$ is generated by polynomials of bounded degree by Proposition \ref{sufficient and necessary conidition}. This contradicts Lemma \ref{lemma7}. 
\end{proof}

\section*{Acknowledgments}

The authors would like to thank the Max Planck Institute f\"ur Mathematik and the first author would like to thank Wayne State University for
their hospitality and excellent working conditions, where some part of this work has been done. 

The second author is supported by the grant AP09261086 of the Ministry of Science and Higher Education of the Republic of Kazakhstan.

\end{document}